\documentclass[12pt,draft]{amsart}
\usepackage{amsxtra,amssymb,amscd,a4wide,amsbsy}
\usepackage[abbrev]{amsrefs}
\usepackage{multirow,colortbl}
\usepackage{tikz} 

\numberwithin{equation}{section}

\newtheorem{theo}{Theorem}[section]
\newtheorem{theorem}{Theorem}[section]
\newtheorem{proposition}[theo]{Proposition}
\newtheorem{lemma}[theo]{Lemma}

\theoremstyle{definition}

\newtheorem{definition}[theo]{Definition}

\theoremstyle{remark}

\newcommand{\Z}{\mathbb Z}
\newcommand{\R}{\mathbb R}
 \newcommand{\ZZ}{\mathsf{Z}}
\newcommand{\A}{\mathscr A}

\renewcommand{\P}{\mathbb P}

\newcommand{\Q}{\mathbb Q}
\newcommand{\eps}{\varepsilon}
\newcommand{\bo}{\boldsymbol{\omega}}

\newcommand{\by}{\boldsymbol{y}}

\DeclareMathAlphabet\mathscr{U}{eus}{m}{n} \SetMathAlphabet\mathscr{bold}{U}{eus}{b}{n} \DeclareMathAlphabet\matheur{U}{eur}{m}{n} \SetMathAlphabet\matheur{bold}{U}{eur}{b}{n}
\def\mybf #1{\textbf{\textit{#1}}}

\begin{document}

\title{Thermodynamics of the Binary Symmetric Channel}

\author{Evgeny Verbitskiy}
\address{Mathematical Institute, Leiden University,
Postbus 9512, 2300 RA Leiden, The  Netherlands \newline\indent \textup{and} \newline\indent Department of
Mathematics, University of Groningen, PO Box 407, 9700 AK Groningen, The Netherlands}
\email{evgeny@math.leidenuniv.nl}
\date{\today}
\begin{abstract} 
We study a hidden Markov process which is the result of
a transmission of
the binary symmetric Markov source over the memoryless binary symmetric channel. 
This process has been studied extensively in Information Theory and is often used as 
a benchmark case for the so-called denoising algorithms.
Exploiting the link between this process
and the 1D Random Field Ising Model (RFIM), we are able to identify the Gibbs potential of the resulting Hidden Markov process. Moreover, we obtain a stronger
bound  on the memory decay rate.
We conclude with a discussion on implications of our results
for the development of denoising algorithms.

\end{abstract}
\maketitle
\section{Introduction}

We study the binary symmetric Markov source over the memoryless binary symmetric channel.
More specifically, let $\{X_n\}$ be a stationary two-state Markov chain with values $\{\pm 1\}$,
and 
$$
\P( X_{n+1}\ne X_n) =p,
$$
where $0<p<1$.  
The binary symmetric channel will be modelled as a sequence of Bernoulli random variables
$\{Z_n\}$ with
$$
\P_Z(Z_n=-1)=\eps,\quad \P_Z(Z_n=1)=1-\eps.
$$
Finally, put  
\begin{equation}\label{model}
Y_n=X_n\cdot Z_n
\end{equation}
for all $n$.
The process $\{Y_n\}$ is a hidden Markov process, because $Y_n\in\{-1,1\}$ is chosen independently for any $n$ from an {\it emission} distribution $\pi_{X_n}$ on $\{-1,1\}$: $\pi_1=(\eps,1-\eps)$ and $\pi_{-1}=(1-\eps,\eps)$. 

The law $\Q$ of the process $\{Y_n\}$ is the push-forward of $\P\times\P_Z$ under $\psi: \{-1,1\}^\Z\times\{-1,1\}^\Z\mapsto \{-1,1\}^\Z$,
with $\psi( (x_n,z_n) )=x_n\cdot z_n$.
We write $\Q=(\P\times\P_Z)\circ\psi^{-1}$. For every $m\le n$, and $y_m^n:=(y_m,\ldots, y_n)\in \{-1,1\}^{n-m+1}$, the measure of the corresponding cylindric set is given by
\begin{equation}\label{eq:cylone}
\aligned
\Q(y_m^n)&:=\Q(Y_m=y_m,\ldots,Y_n=y_n)\\
& =\sum_{x_{m}^n,z_{m}^n\in\{-1,1\}^{n-m+1}} \P(x_m^n)\P_Z(z_{m}^n)
\prod_{k=m}^n \mathbb I[ y_k=x_k\cdot z_k]\\
&=\sum_{x_{m}^n\in\{-1,1\}^{n-m+1}} \frac 12 \prod_{i=m}^{n-1} p_{x_i,x_{i+1}} \cdot
\eps^{\#\{i\in [m,n]: x_i y_i=-1\}}(1-\eps)^{\#\{i\in [m,n]: x_i y_i=1\}}.
\endaligned
\end{equation}

\section{Random Field Ising Model}\label{sec:RFIM}
It was observed in \cite{Zuk} that  the probability $\Q(y_{m},\ldots,y_n)$ of a cylindric event $\{Y_m=y_m,\ldots, Y_n=y_n\}$, $m\le n$,
can be expressed via a partition function of a random field Ising model.
We exploit this observation further.
Assume $p>0$ and $\eps>0$, and put
$$
J=\frac 12\log\frac {1-p}{p},\quad K=\frac 12\log \frac {1-\eps}{\eps}.
$$
Then for any $(y_m,\ldots,y_n)\in\{-1,1\}^{n-m+1}$, expression for 
the cylinder probability (\ref{eq:cylone}) can be rewritten as 
$$
\Q(y_m,\ldots,y_n)= 
 \frac {c_J}{\lambda_{J,K}^{n-m+1}}
\sum_{x_m^n\in\{-1,1\}^{n-m+1}}\exp\Bigl( J\sum_{i=m}^{n-1} x_ix_{i+1} +K\sum_{i=m}^n x_i y_i\Bigr),
$$
where
$$\aligned
c_J&= {\cosh(J)},\quad
\lambda_{J,K}&=2\left(  \cosh(J+K)+ \cosh(J-K)\right)=4\cosh(J)\cosh(K). \\
\endaligned
$$
The non-trivial part of the cylinder probability  is the sum over all hidden configurations $(x_m,\ldots,x_n)$:
$$
\ZZ_{n,m}(y_n^m):=
\sum_{x_m^n\in\{-1,1\}^{n-m+1}}\exp\Bigl( J\sum_{i=m}^{n-1} x_ix_{i+1} +K\sum_{i=m}^n x_i y_i\Bigr)
$$
is in fact the partition function of the Ising model with the random field given by $y$'s.
Applying  the recursive method of \cite{Rujan}, the partition function can be 
evaluated in the following fashion \cite{BZ}. Consider the following functions
$$\aligned
A(w) &=\frac 12 \log \frac{\cosh(w+J)}{\cosh(w-J)},\\
B(w) &=\frac 12 \log\Bigl[ 4\cdot{\cosh(w+J)}{\cosh(w-J)}\Bigr]= \frac 12 \log\Bigl[ e^{2w}+e^{-2w}+e^{2J}+e^{-2J}\Bigr]
\endaligned
$$
One readily checks that if $s=\pm 1$, then for all $w\in\R$
\begin{equation}\label{basic}
\exp\Bigl( s A(w)+B(w) \Bigr) =2\cosh( w+s J).
\end{equation}
Now the partiton function   can be evaluated 
by summing the right-most spin. 
Namely, suppose $m<n$, $y_m^n\in \{-1,1\}^{n-m+1}$,  then
$$\aligned
\ZZ_{m,n}(y_m^n)
&=\sum_{x_m^{n-1}\in\{-1,1\}^{m-n}} 
\exp\Bigl( J\sum_{i=m}^{n-2} x_{i}x_{i+1}+
K\sum_{i=m}^{n-1} x_i y_i\Bigr)\sum_{x_n\in\{-1,1\}} e^{x_n( Jx_{n-1}+Ky_n)}
\\
&=\sum_{x_m^{n-1}\in\{-1,1\}^{m-n}} 
\exp\Bigl( J\sum_{i=m}^{n-2} x_{i}x_{i+1}+
K\sum_{i=m}^{n-1} x_i y_i\Bigr) \Bigl\{ 2\cosh( Jx_{n-1}+Ky_n) \Bigr\}\\
&=\sum_{x_m^{n-1}\in\{-1,1\}^{m-n}} 
\exp\Bigl( J\sum_{i=m}^{n-2} x_{i}x_{i+1}+
K\sum_{i=m}^{n-1} x_i y_i\Bigr) \exp\Bigl(x_{n-1}A(w_n^{(n)})+B(w_n^{(n)})\Bigr)
\endaligned
$$
where
$$
w_n^{(n)} = Ky_n.
$$
Hence,
$$\aligned
\ZZ_{m,n}(y_m^n)&=\sum_{x_m^{n-1}\in\{-1,1\}^{m-n}} 
\exp\Bigl( J\sum_{i=m}^{n-2} x_{i}x_{i+1}+
K\sum_{i=m}^{n-2} x_i y_i  +x_{n-1}\bigl(\underbrace{K y_{n-1}+A(w_n^{(n)})}_{w_{n-1}^{(n)}}\bigr)\Bigr) \\
&\qquad\times\exp\Bigl(B(w_n^{(n)})\Bigr).
\endaligned
$$
and thus the new sum has exactly the same form, but instead of $Ky_{n-1}$, we now have 
$w_{n-1}^{(n)}=Ky_{n-1}+A(w_n^{(n)})$. Continuing the summation over the remaining right-most $x$-spins, one gets
$$\aligned
\ZZ_{m,n}(y_m^n)&=2\cosh( w_m^{(n)})\exp\Bigl( \sum_{i=m+1}^n B(w_i^{(n)})\Bigr) ,
\endaligned
$$
where 
$$
 w_i^{(n)}=Ky_i+A(w_{i+1}^{(n)})\quad\text{ for every } i<n,
$$
equivalently, since $A(0)=0$,  we can define
$$
w_{i}^{(n)}=0 \quad\forall i>n, \ \text{ and }\ w_i^{(n)}=Ky_i+A(w_{i+1}^{(n)})\quad\forall i\le n.
$$
Therefore, we obtain the following expressions for the cylinder and conditional probabilities 
\begin{equation}\label{express}
\aligned
\Q(y_0^n) &=\frac {c_J}{\lambda^{n+1}_{J,K}} \cosh(w_0^{(n)})\exp\left(\sum_{i=1}^n B(w_i^{(n)})\right),\\
\Q(y_0|y_1^n)&=\frac 1{\lambda_{J,K}}\frac{\cosh(w_0^{(n)})\exp\left( B(w_1^{(n)})\right)}{\cosh(w_1^{(n)})}.
\endaligned
\end{equation}

\section{Thermodynamic formalism}

Let $\Omega=A^{\mathbb Z_+}$, where $A$ is a finite alphabet,
be the space of one-sided infinite sequences $\bo=(\omega_0,\omega_1,\ldots)$ in alphabet $A$ ( $\omega_i\in A$ for all $i$).
We equip $\Omega$ with the metric
$$
d(\bo ,\tilde\bo)=2^{-k(\bo ,\tilde\bo)},
$$
where $k(\bo ,\tilde\bo)=1$ if $\omega_0\ne\tilde\omega_0$,
and $k(\bo ,\tilde\bo) = \max\{k\in\mathbb N: \omega_{i}=\tilde\omega_i\ 
\quad \forall i=0,\ldots,k-1\}$, otherwise. Denote by $S:\Omega\to\Omega$
the left shift: 
$$
 (S\bo)_{i} = \omega_{i+1}\text{ for all }i\in\Z_+.
$$
Borel probability measure $\P$ is translation invariant if 
$$
\P(S^{-1}C) =\P(C)
$$
for any Borel event $C\subseteq \Omega$.

Let us recall the following well-known definitions:

\begin{definition}\label{DefGibbs} Suppose $\P$ is a fully supported translation invariant measure on $\Omega=A^{\mathbb Z_+}$, where $A$ is a finite alphabet.

(i) The measure $\P$ is called a \mybf{$g$-measure}, if for some positive continuous  function $g:\Omega\to (0,1)$
satisfying the normalization condition
$$
\sum_{{\bar\omega}_0\in A} g({\bar\omega}_0,\omega_1,\omega_2,\ldots)=1
$$
for all $\bo=(\omega_0,\omega_1,\ldots)\in\Omega$, one has
$$
\P(\omega_0|\omega_1,\omega_2,\ldots) = g(\omega_0,\omega_1,\ldots)
$$
for $\P$-a.a. $\bo\in\Omega$.

(ii) The measure $\P$ is \mybf{Bowen-Gibbs} for a continuous potential $\phi:\Omega\to\mathbb R$,
if there exist constants $P\in \R$ and $C\ge 1$ such that for all $\mathbf \omega\in\Omega$
and every $n\in \mathbb N$
$$
  \frac 1C \le \frac 
  {\P(
  \{
 \tilde{\bo}\in\Omega:\, \, {\tilde\omega}_0=\omega_0,\ldots {\tilde\omega}_{n-1}=\omega_{n-1}
  \}
  )
  }
  {\exp\bigl(  (S_n\phi)(\bo)-nP \bigr)}\le C,
$$
where $(S_n\phi)(\bo)=\sum_{k=0}^{n-1} \phi(\sigma^k \bo)$.

(iii) The measure $\P$ is called an \mybf{equilibrium state} for continuous
potential $\phi:\Omega\to \R$, if $\P$ attains maximum of the following
functional
\begin{equation}\label{eq:varprin}
h(\P)+\int \phi \, d\P = \sup_{\tilde \P\in\mathcal M_1^*(\Omega)}\Bigl[h(\tilde\P)+\int \phi \, d\tilde \P \Bigr],
\end{equation}
where $h(\P)$ is the Kolmogorov-Sinai entropy of $\P$ and
the supremum is taken over the set $\mathcal M_1^*(\Omega)$ of all translation invariant
Borel probability measures on $\Omega$.
\end{definition}

It is known that every $g$-measure $\P$ is also an equilibrium state
for $\log g$; and that every Bowen-Gibbs measure  $\P$ for potential $\phi$
is an equilibrium state for $\phi$ as well.

\begin{theorem}\label{l:gmeas} The measure $\Q$  on $\{-1,1\}^{\Z_+}$ (c.f., \eqref{express})
is a $g$-measure   for some positive continuous function $g$ with an \emph{exponential decay 
of variation}:
\begin{equation}\label{estimdecay1}
\text{\rm var}_n(g) := \sup_{\by,\tilde\by: y_0^{n-1}={\tilde y}_{0}^{n-1}} \bigl| g(\by)-g(\tilde{\by})\bigr|
\le C\rho^n,
\end{equation}
where $C>0$ and $\rho\in (0,1)$. The measure $\Q$ is also a Bowen-Gibbs measure for a H\"older
continuous potential $\phi:\{-1,1\}^{\Z_+}\to\mathbb R$.
\end{theorem}

The result of Theorem \ref{l:gmeas} is actually true in much greater generality: namely,  
for distributions of Hidden Markov Chains $\{Y_n\}$,
where the underlying Markov chain $\{X_n\}$ has
strictly positive transition probability matrix $P$, see  \cite{Verb}
for review of several results of this nature. 
However, the present situation is rather exceptional since one is able to identify the $g$-function and
 the Gibbs potential $\phi$ explicitly.   Another  interesting question is the estimate of the decay rate $\rho$.
 In \cite{Verb} a number of previously known estimates
of the rate of exponential decay in (\ref{estimdecay1}) have been compared; 
the best known estimate for $\rho$ 
$$\rho\le |1-2p|
$$  is due to \cite{HJ} and \cite{Fernandez}. 
Quite surprisingly the estimate does not depend on $\eps$,
and in fact, it was conjectured in \cite{Verb} that
the estimate could be improved, e.g., by incorporating dependence on $\eps$.
The proof of Theorem \ref{l:gmeas} shows that this is indeed the case and one obtains
a new estimate
$$
\rho\le \rho^*(p,\eps)<|1-2p|.
$$

We start with the following technical result.

\begin{lemma}\label{l:decay} Fix $\by=(y_0,y_1,\ldots)\in\{-1,1\}^{\Z_+}$. For every $n\in \Z_+$,
define the sequence $w^{(n)}_i=w^{(n)}_i(\by)$, $i\in\Z_+$,  by letting $w^{(n)}_i=0$ for every $i\ge n+1$
and $w_i^{(n)}=Ky_i+A(w_{i+1}^{(n)})$ for $i\le n$. Then for every $i\in\Z_+$
$$
\lim_{n\to\infty} w_i^{(n)} =:w_{i}(\by).
$$
Moreover, there exist constants 
$\varrho\in (0,1)$ and $C>0$, both independent of $\by$, such that
\begin{equation}\label{eq:decayest}
 |w_i^{(n)}(\by)-w_i(\by)| \le C\varrho^{n-i}
\end{equation}
for all $n\ge i$, and therefore, $w_i:\{-1,1\}^{\Z_+}\to \R$
is H\"older continuous for every $i\in\Z_+$:
$$
| w_i(\by)-w_i(\tilde\by)|= | w_0(S^i\by)-w_0(S^i\tilde\by)| \le C' \left(d(S^i\tilde\by,S^i\tilde\by)\right)^\theta
$$
for some $C',\theta>0$ and all $ \by,\tilde\by\in\{-1,1\}^{\Z_+}$.
\end{lemma}
 \begin{proof}  Suppose $i\le n\le m$. Then 
$$
|w_i^{(n)}-w_i^{(m)}| =| A(w_{i+1}^{(n)})-A(w_{i+1}^{(m)})|\le
| w_{i+1}^{(n)}-w_{i+1}^{(m)}|\cdot \sup_{w}\Bigl|\frac {dA}{dw} \Bigr|.
$$
One has
$$
\frac {dA}{dw}=\frac {\sinh(2J)}{\cosh(2J)+\cosh(2w)}
$$
and hence
\begin{equation}\label{e:basic}
 \varrho:=\sup_{w}  \Bigl|\frac {dA}{dw}\Bigr|=\Bigl|\frac {\sinh(2J)}{\cosh(2J)+1}\Bigr|=|\tanh(J)|= |1-2p|<1.
\end{equation}
Combined with the fact that for all $i\in\Z_+$
$$
|w_i^{(m)}|=|Ky_i+A(w_{i+1}^{(m)})| \le |K| +|\text{arctanh}(1-2p)|\le |K|+|J|=:C_1.
$$
Therefore for $i\le n\le m$
$$
|w_i^{(n)}-w_i^{(m)}|\le  \varrho^{n-i} |w_{n+1}^{(n)}-w_{n+1}^{(m)}| =
 \varrho^{n-i} |w_{n+1}^{(m)}| \le C_1\varrho^{n-i}.
$$
Hence, $\lim_{n\to\infty} w_i^{(n)}=:w_i$ exists and
$$
|w_i^{(n)}-w_i|\le \sum_{m=n}^\infty |w_i^{(m)}-w_i^{(m+1)}|\le
C_1\sum_{m=n}^\infty \varrho^{m-i}= \frac {C_1}{1-\varrho} \rho^{n-i}=:C\varrho^{n-i}.
$$
The estimate in (\ref{e:basic}) can be improved.
Firstly, assume that $\eps<p$. In this case, $|K|>|J|>0$, and if $i\le n$, then
$$
w_i^{(n)} \in \bigl[ -|K|-|J|,-|K|+|J|\bigr] \cup \bigl[ |K|-|J|, |K|+|J|\bigr],
$$
i.e., $|w_i^{(n)}|$ is bounded away from $0$ (see Figure \ref{fig:graphs}.(a)). Therefore, we 
can define $\varrho$ by
$$\aligned
\varrho &=\varrho(J,K)=\varrho(p,\eps) =\sup_{w\in [|K|-|J|,|K|+|J|]} \Bigl| \frac {dA}{dw}\Bigr|=
\biggl|\frac{\sinh(2J)}{\cosh(2J)+\cosh(2K-2J)}\biggr|\\
&=\frac {\eps(1-\eps)|1-2p|}{p^2-2p\eps+\eps}=\frac {\eps(1-\eps)}{(p-\eps)^2+\eps(1-\eps)}|1-2p|<|1-2p|.
\endaligned
$$
 \pgfmathsetmacro{\e}{1.44022}   
\begin{figure}[ht]
\begin{tikzpicture}
  \pgfmathsetmacro{\J}{0.8}
  \pgfmathsetmacro{\K}{1.5}
  \draw[->] (-3,0) -- (3.5,0) node[right] {$w$};
  \draw[->] (0,-2.3) -- (0,2.5);
  \draw[scale=0.95,domain=-3:3.5,smooth,variable=\x,blue,line width=2pt] 
     plot ({\x},{\K+1/2*ln(cosh(\x+\J)/cosh(\x-\J))});
  \draw[scale=0.95,domain=-3:3.5,smooth,variable=\x,red,line width=2pt] 
     plot ({\x},{-\K+1/2*ln(cosh(\x+\J)/cosh(\x-\J))});
\draw[scale=0.95,domain=-3:0,dotted,variable=\x,red,line width=1pt] 
     plot ({\x},{-\K-\J)});
\draw[scale=0.95,domain=0:3,dotted,variable=\x,red,line width=1pt] 
     plot ({\x},{-\K+\J)});
\draw[scale=0.95,domain=-3:0,dotted,variable=\x,blue,line width=1pt] 
     plot ({\x},{\K-\J)});
\draw[scale=0.95,domain=0:3, dotted,variable=\x,blue,line width=1pt] 
     plot ({\x},{\K+\J)});
 \node at (2.4,1.4) {$\scriptstyle F_+(w)=K+A(w)$};
 \node at (2.4,-1.2) {$\scriptstyle F_-(w)=-K+A(w)$};
  \node at (0.4,0.7) {$\scriptstyle K-J$};
  \node at (-0.6,2.2) {$\scriptstyle K+J$};
  \node at (-0.6,-0.6) {$\scriptstyle -K+J$};
  \node at (0.6,-2.2) {$\scriptstyle -K-J$};
 
\draw[scale=0.95,domain=0:1,variable=\x,black,line width=2pt] 
     plot ({0},{(\K-\J)*\x+(\K+\J)*(1-\x)});
\draw[scale=0.95,domain=0:1,variable=\x,black,line width=2pt] 
     plot ({0},{(-\K+\J)*\x+(-\K-\J)*(1-\x)});
\end{tikzpicture}
\begin{tikzpicture}
  \pgfmathsetmacro{\J}{1.4}
  \pgfmathsetmacro{\K}{0.9}
  \draw[->] (-3,0) -- (3.5,0) node[right] {$w$};
  \draw[->] (0,-2.3) -- (0,2.5);
  \draw[scale=0.95,domain=-3:3.5,smooth,variable=\x,blue,line width=2pt] 
     plot ({\x},{\K+1/2*ln(cosh(\x+\J)/cosh(\x-\J))});
  \draw[scale=0.95,domain=-3:3.5,smooth,variable=\x,red,line width=2pt] 
     plot ({\x},{-\K+1/2*ln(cosh(\x+\J)/cosh(\x-\J))});
\draw[scale=0.95,domain=-3:0,dotted,variable=\x,red,line width=1pt] 
     plot ({\x},{-\K-\J)});
\draw[scale=0.95,domain=0:3,dotted,variable=\x,red,line width=1pt] 
     plot ({\x},{-\K+\J)});
\draw[scale=0.95,domain=-3:0,dotted,variable=\x,blue,line width=1pt] 
     plot ({\x},{\K-\J)});
\draw[scale=0.95,domain=0:3, dotted,variable=\x,blue,line width=1pt] 
     plot ({\x},{\K+\J)});
 \node at (2.4,1.4) {$\scriptstyle F_+(w)=K+A(w)$};
 \node at (2.4,-1.2) {$\scriptstyle F_-(w)=-K+A(w)$};
  \node at (0.4,0.3) {$\scriptstyle K-J$};
  \node at (-0.6,2.2) {$\scriptstyle K+J$};
  \node at (-0.6,-0.7) {$\scriptstyle -K+J$};
  \node at (0.6,-2.2) {$\scriptstyle -K-J$};
 
\draw[scale=0.95,domain=0:1,variable=\x,black,line width=2pt] 
     plot ({0},{(\K-\J)*\x+(\K+\J)*(1-\x)});
\draw[scale=0.95,domain=0:1,variable=\x,black,line width=2pt] 
     plot ({0},{(-\K+\J)*\x+(-\K-\J)*(1-\x)});
\end{tikzpicture}
\caption{Graphs of $F_+(w)=K+A(w)$, $F_-(w)=-K+A(w)$ for (a) $\eps<p\le 0.5$ and (b) $p\le\eps\le 0.5$. }\label{fig:graphs}
\end{figure}

If $\eps>p$ (equivalently, $K<J$), then the maps $F_+(w)=K+A(w)$ and $F_-(w)=-K+A(w)$
have no longer disjoint images  (c.f.,  Figure \ref{fig:graphs}.(b)) . Nevertheless, one can consider 
second iterations:  
$$\aligned
|w_i^{(n)}-w_i^{(m)}| &=| A(w_{i+1}^{(n)})-A(w_{i+1}^{(m)})|=\bigl|
A\bigl(Ky_{i+1}+A(w_{i+2}^{(n)})\bigr)-
A\bigl(Ky_{i+1}+A(w_{i+2}^{(m)})\bigr)\bigr|\\
&\le  \bigl(  \sup_{w} |A'(K+A(w)) A'(w)|\bigr) =:\rho^{(2)} |w_{i+2}^{(n)}-w_{i+2}^{(m)}|
\endaligned
$$
One can show that 
\begin{equation}\label{eq:secderone}
\rho^{(2)}=\sup_{w} |A'(K+A(w)) A'(w)| <(1-2p)^2.
\end{equation}
Informally, it is evident that  the maximal value of the derivative $|A'(\cdot)|$, equal to $|1-2p|$,
is attained, if $w=0$ or if $K+A(w)=0$,
but then $K+A(w)\ne 0$ or $w\ne 0$, respectively, and hence (\ref{eq:secderone}) holds.
Similar argument generalises to all $w$: firstly,  note that
\begin{equation}\label{eq:secder}
 |A'(K+A(w)) A'(w)|  = \frac {(1-2p)^2}
 {\bigl(\alpha+(1-\alpha)\cosh(2K +2A(w))\bigr)
 \cdot  \bigl(\alpha+(1-\alpha)\cosh(2w)\bigr)},
\end{equation}
where $\alpha=(1-p)^2+p^2$, $1-\alpha=2p(1-p)$.
Let $\Delta>0$ be such that for all 
$w\in [-\Delta,\Delta]$ one has $|A(w)|<|K|/2$  and   $\cosh(2K+2A(w))>\cosh(K)$, 
and hence
$$
 |A'(K+A(w)) A'(w)|  \le \frac {(1-2p)^2}
 {\bigl(\alpha+(1-\alpha)\cosh(K)\bigr)\cdot 1}<(1-2p)^2.
$$
For $w\not\in [-\Delta,\Delta]$, one  has
$$
 |A'(K+A(w)) A'(w)|  \le \frac {(1-2p)^2}
 {1\cdot \bigl(\alpha+(1-\alpha)\cosh(\Delta)\bigr)}<(1-2p)^2.
 $$
 Hence, 
 $$
 \rho^{(2)} =\min\left\{\frac {(1-2p)^2}
 {\bigl(\alpha+(1-\alpha)\cosh(K)\bigr)}, \frac {(1-2p)^2}
 {\bigl(\alpha+(1-\alpha)\cosh(\Delta)\bigr)}\right
 \}<(1-2p)^2,
 $$
 and hence $\varrho=\sqrt{\rho^{(2)}}<|1-2p|$.
 Sharper bounds can be achieved by studying minimum of
the denominator in (\ref{eq:secder}).
\end{proof}

\begin{proof}[Proof of Theorem \ref{l:gmeas}]
To show that $\Q$ is a $g$-measure it is sufficient to 
show that conditional probabilities $\Q(y_0|y_1^n)$ converge
uniformly as $n\to\infty$. Given that
\begin{equation}\label{eq:condprobtwo}
\Q(y_0|y_1^n)=\frac 1{\lambda_{J,K}}\frac{\cosh(w_0^{(n)})\exp\left( B(w_1^{(n)})\right)}{\cosh(w_1^{(n)})},
\end{equation}
and using the result of Lemma \ref{l:decay}:  $
w_i^n(\by) \rightrightarrows w_i(\by)$ as $n\to\infty$,
we obtain uniform convergence of conditional probabilities,
and hence, $\Q$ is a $g$-measure with $g$ given by
\begin{equation}\label{eq:gfunct}
g(\by) =\frac 1{\lambda_J}\frac{\cosh(w_0(\by))\exp\left( B(w_1(\by))\right)}{\cosh\left(w_1(\by)\right)}.
\end{equation}
Let us introduce the following functions: for $\by\in\{-1,1\}^{\Z_+}$, put
$$
\phi(\by) =  B(w_0(\by)),\quad h(\by) =\cosh( w_0(\by))\exp\left( -B(w_0(\by))\right).
$$
Taking into account that $w_1(\by)=w_0(S\by)$, one has
$$
g(\by) =\frac {e^{\phi(\by)}}{\lambda_{J,K}} \frac {h(\by )}
{h(S\by)}.
$$
Since every $g$-measure is also an equilibrium state for $\log g$,
we conclude that $\Q$ is an equilibrium state for 
$$
\tilde \phi(\by)=\phi(\by) +\log h(\by)-\log h(S\by) -\log \lambda_{J,K}.
$$
The difference $\tilde \phi(\by)-\phi(\by)$ has a very special form: it is a sum of a so-called coboundary ($\log h(\by)-\log h(S\by)$)
and a constant ($-\log \lambda_{J,K}$). Two potentials whose difference
is of a such  form, have identical sets of equilibrium states. The reason is that for any translation invariant measure $\Q'$ one has
$$
\int\bigl( \log h(\by)-\log h(S\by)-\log\lambda_{J,K}\bigr)d\Q' =-\log
\lambda_{J,K}=\text{const}.
$$
Therefore, if $\Q'$ achieves maximum in the righthand side of (\ref{eq:varprin}) for $\tilde\phi$, then  $\Q'$ achieves maximum
for $\phi$ as well. Thus $\Q$ is also an equilibrium state for 
$$
\phi(\by) = B(w_0(\by))= \frac 12\log\Bigl[ 4\sinh^2(w_0(\by))+ \frac 1{p(1-p)}\Bigr].
$$
Any equilibrium measure for a H\"older continuous potential $\phi$
is also a Bowen-Gibbs measure \cite{BowenGibbs}.
In our particular case,  direct proof of the Bowen-Gibbs property for $\Q$
is straightforward. Indeed,
using the result of \eqref{express} and the notation introduced above, for every $\by=(y_0,y_1,\ldots)$ one has
$$\aligned
\Q(y_0^n) &= \frac {c_J}{\lambda^{n+1}_{J,K}}
\exp\Bigl( \sum_{i=1}^{n}B(w_i^{(n)}(\by))\Bigr)\cosh( w_0^{(n)}(\by))\\
&=\frac {c_{J}\cdot\cosh(w_0^{(n)}(\by))}{\exp(B(w_0(\by)))} \exp\Bigl( \sum_{i=1}^n [B(w_i^{(n)}(\by))-
B(w_i(\by))]\Bigr)\\
&\quad\times\exp\Bigl( \sum_{i=0}^{n}B(w_i(\by))-(n+1)\log\lambda_{J,K}\Bigr).
\endaligned
$$
Therefore, for $P=\log \lambda_{J,K}$, 
$$\aligned
\frac {\Q(y_0^n)}{\exp\bigl( (S_{n+1}\phi)(\by)-(n+1)P \bigr)}&=
 \frac {c_{J}\cdot\cosh(w_0^{(n)}(\by))}{\exp(B(w_0(\by)))}  \exp\Bigl( \sum_{i=1}^n [B(w_i^{(n)}(\by))-
B(w_i(\by))]\Bigr)
\endaligned
$$
It only remains to demonstrate 
that the right hand side is 
 uniformly bounded  (both in $n$ and $\by=(y_0,y_1,\ldots)$)
from below and above by some positive constants
$\underline C,\overline C$, respectively.
Indeed, since $p,\eps>0$, $I=[-|K|-|J|,|K|+|J|]$ is a finite interval, by the result of the previous Lemma,
$w_i^{(n)}(\by)\in I$
for all $i$ and $n$. Using \eqref{eq:decayest}, one readily checks that the following choice of constants suffices:
$$\aligned
\overline C&= c_J\frac {\sup_{w\in I} \cosh(w)}{\inf_{w\in I} \exp(B(w))}
\exp\biggl( \frac C{1-\varrho} \sup_{w\in I} \Bigl|\frac {dB}{dw}\Bigr| \biggr)<\infty,\\ 
\underline C&=c_J\frac {\inf_{w\in I} \cosh(w)}{\sup_{w\in I}  \exp(B(w))}
\exp\biggl(-\frac C{1-\varrho} \sup_{w\in I} \Bigl|\frac {dB}{dw}\Bigr| \biggr)>0.
\endaligned
$$
\end{proof}
We complete this section with a curious continued fraction
 representation of the $g$-function (\ref{eq:gfunct}).
 
 \begin{proposition} For every $\by =(y_0,y_1,\ldots)\in\{-1,1\}^{\Z_+}$, one has
  
\begin{equation}\label{eq:continued}
2g(\by)= 
a_1
-\cfrac{b_1}{a_2-\cfrac{b_2}{a_3-\cfrac{b_3}{a_4-\ldots}}}
$$
where for $i\ge 1$
$$\aligned
q_i&=(1-2p)y_{i-1} y_{i},\quad
a_i&= 1+q_i, \quad b_i= 4\eps (1-\eps) q_i.
\endaligned
\end{equation}
 \end{proposition}
 \begin{proof} Using elementary transformations, one can show that
 for every $\by=(y_0,y_1,\ldots)\in\{-1,1\}^{Z_+}$ one has
 \begin{equation}\label{express3}
\aligned
g(\by)&=
\frac 1{\lambda_{J,K}} \frac {\cosh(w_0(\by))}{\cosh(w_1(\by))} \exp\Bigl( B(w_1(\by))\Bigr)\\
&=\frac 12 +\frac 12 (1-2p)(1-2\eps)y_0\tanh(w_1(\by))\\
\endaligned
\end{equation}
 Since
$$
\tanh( A(w)) =\tanh(J)\tanh(w)=(1-2p)\tanh(w) \quad \text{for all } w\in\R,
$$ 
for every $i\in \Z_+$, one has
$$\aligned
\tanh(w_i)&=\frac {\tanh(Ky_i)+\tanh(A(w_{i+1}))}{1+\tanh(Ky_i)\cdot\tanh(A(w_{i+1}))}
\\
&=\frac{ (1-2\eps)y_i+(1-2p)\tanh(w_{i+1})}{1+(1-2\eps)(1-2p)y_{i}\tanh(w_{i+1})}\\
&=y_i\frac{ (1-2\eps)+(1-2p)y_i\tanh(w_{i+1})}{1+(1-2\eps)(1-2p)y_{i}\tanh(w_{i+1})},
\endaligned
$$
Therefore, if we let $z_{i}=(1-2p)(1-2\eps)y_{i-1}\tanh(w_i)$, $i\in \mathbb N$, then
$$\aligned
z_{i}&=(1-2p)y_{i-1}y_{i} -\frac {4\eps(1-\eps)(1-2p)y_{i-1}y_{i}} {1+z_{i+1}}
=q_i-\frac{ b_i}{1+z_{i+1}}.
\endaligned
$$ 
Since $g(\by) =\frac 12 +\frac 12z_1$, we obtain the continued fraction expansion \eqref{eq:continued}
\end{proof}

\section{Two-sided conventional probabilities and denoising}

In the previous section we established that $\Q$ is a Bowen-Gibbs measure. The notion of a Gibbs measure originates in Statistical Mechanics, and is not equivalent
to the Bowen-Gibbs definition.
In Statistical Mechanics, one is interested in \mybf{two-sided} conditional probabilities
$$
\Q(y_0|y_{-m}^{-1},y_1^n) \text{ or }
\Q(y_0|y_{<0},y_>0) := \Q(y_0|y_{-\infty}^{-1},y_{1}^\infty).
$$
The method of section 2 can be used to evaluate continual
probabilities $\Q(y_0|y_{-m}^{-1},y_1^n)$, $m,n>0$ for
$\by=(\ldots,y_{-1},y_0,y_{1},\ldots)\in\{-1,1\}^{\Z}$. Indeed,
$$
\Q(y_0|y_{-m}^{-1},y_1^n)= \frac{\Q(y_{-m}^{-1},y_0,y_1^n)}{\Q(y_{-m}^{-1},y_0,y_1^n)+\Q(y_{-m}^{-1},\bar y_0,y_1^n)},
$$
where $\bar y_0=-y_0$. We can evaluate
$$\aligned\Q(y_{-m},\ldots,y_{-1},y_0,y_1,\ldots,y_n)&= 
 \frac {c_J}{\lambda_{J,K}^{n+m+1}}
\sum_{x_{-m}^n\in\{-1,1\}^{n+m+1}}\exp\Bigl( J\sum_{i=-m}^{n-1} x_ix_{i+1} +K\sum_{i=-m}^n x_i y_i\Bigr)\\
&=\frac {c_J}{\lambda_{J,K}^{n+m+1}} {\mathsf Z}_{-m,n}(y_{-m}^n),
\endaligned
$$
by first summing over  spins on the right: $x_n,\ldots,x_{1}$,
and then summing over spins on the left: $x_{-m},\ldots,x_{-1}$.  One has
$$
\aligned
{\mathsf Z}_{-m,n}(y_{-m}^n) =&\sum_{x_{-m},\ldots,x_{0}} 
\exp\Bigl( 
J\sum_{i=-m}^{-1} x_{i}x_{i+1}+
K\sum_{i=m}^{0} x_i y_i+x_0A( w_{1}^{(n)}) \Bigr) 
\exp\left(\sum_{i=1}^n B(w_i^{(n)})\right)\\
&=\exp\left(\sum_{j=-m}^{-1} B(w_j^{(-m)})\right)2\cosh\bigl(  w_0^{(-m,n)}\bigr)  
\exp\left(\sum_{i=1}^n B(w_i^{(n)})\right)\\
\endaligned
$$
where now $w_{-m}^{(-m)}=Ky_{-m}$,
$$
 w_{j+1}^{(-m)} = Ky_{j+1} +A(w_{j}^{(-m)}),\quad j=-m,\ldots,-2.
$$
and
$$
w_0^{(-m,n)} = Ky_0+A(w_{-1}^{(-m)})+A(w_0^{(n)}).
$$
Therefore,
$$\aligned
\Q(y_0|y_{-m}^{-1},y_1^n)&= \frac{{\mathsf Z}_{-m,n}(y_{-m}^{-1},y_0,y_1^n)}{{\mathsf Z}_{-m,n}(y_{-m}^{-1},y_0,y_1^n)+{\mathsf Z}_{-m,n}(y_{-m}^{-1},\bar y_0,y_1^n)}\\
&= \frac { \cosh\bigl(Ky_0+A(w_{-1}^{(-m)})+A(w_1^{(n)}\bigr)}
{ \cosh\bigl(Ky_0+A(w_{-1}^{(-m)})+A(w_1^{(n)})\bigr)+{ \cosh\bigl(-Ky_0+A(w_{-1}^{(-m)})+A(w_1^{(n)})\bigr)}.
}
\endaligned
$$
Again, given this expression, one easily establishes uniform convergence and existence of the limits, 
$$
\Q(y_0|y_{-\infty}^{-1},y_1^\infty)=\lim_{m,n\to\infty} \Q(y_0|y_{-m}^{-1},y_1^n).
$$
Thus the two sided conditional probabilities are also regular, c.f. Theorem \ref{l:gmeas}.

\subsection{Denoising}  Reconstruction of signals corrupted by noise during the transmission is one of the classical problems
in Information Theory. Suppose we observe a sequence $\{y_n\}$, $n=1,\ldots, N$, given by (\ref{model}), i.e.,
$$
y_n=x_n\cdot z_n.
$$ 
where $\{x_n\}$ is some unknown realisation of the   Markov chain,
 and $\{z_n\}$ is  unknown realisation of  the Bernoulli sequence $\{Z_n\}$. The natural
question is, given the observed data $y^N=(y_1,\ldots, y_N)$,
what is the optimal choice of $\hat X_n=\hat X_n(y^N)$ -- the estimate of $X_n$, such that
the empirical zero-one loss (bit error rate)
$$
L_N=\frac 1N\sum_{n=1}^{N} \mathbb I[ \hat X_n\ne  x_n ].
$$
is minimal.
The corresponding standard maximum a posteriori probability (MAP) estimator  (denoiser)
 is given by
$$
\hat X^n =\hat X^n(y^N)
= \mathop{\mathrm{argmax}}\limits_{x\in\{-1,1\}} \ \P[X_n=  x\,|\,Y^N=y_1^N],
\quad n=1,\ldots,N.
$$
In case, parameters of the Markov chain (i.e., $P$) and of the channel (i.e., $\Pi$)
are known, conditional probabilities $\P[X_n=  x\,|\,Y^N=y^N]$ can be 
found using the \mybf{backward-forward algorithm}. Namely,  one has 
\begin{equation}\label{BF}
\P[X_n=  x\,|\,Y^N=y^N] = \frac {\alpha_n(x)\beta_n(x)}{
\sum_{\tilde x\in \A} \alpha_n(\tilde x)\beta_n(\tilde x)}
\end{equation}
where  
$$\aligned
\alpha_{n}(x)&=\P[Y_1^n=y_1^n, X_n=x],\quad \beta_n(x)=\P[Y_{n+1}^N=y_{n+1}^N|X_n=x]
\endaligned
$$
are the so-called  forward and backward variables, 
satisfying simple recurrence relations:
$$\aligned
\alpha_{n+1}(x)&= \sum_{\tilde x\in A} \alpha_n(\tilde x)\ P_{\tilde x,x}\ \Pi_{x,y_{n+1}},\  n=1,\ldots, N-1,\text{ with } \alpha_1(x) = \P(X_1=x)\Pi_{x,y_1},\\ \beta_n(x)&=\sum_{\tilde x\in A} \beta_{n+1}(\tilde x)\ P_{x,\tilde x}\ \Pi_{\tilde x,y_{n+1}},\ n=1,\ldots, N-1, \text{ with } \beta_N(x) =1.
\endaligned
$$
The key observation of \cite{Weissman} is that the probability
distribution $\P[X_n=\cdot\ | Y^N=y^N]$, viewed as a column vector,
can be expressed in terms of two-sided
conditional probabilities \\ $\Q[Y_n=\cdot\ | Y^{N\setminus n} = y^{N\setminus n}]$,
with $N\setminus n=\{1,\ldots,N\}\setminus\{n\}$, as follows
\begin{equation}\label{TWOSIDED}
\P[X_n=\cdot\ | Y^N=y^N] = \frac {\pi_{y_n}\odot \Pi^{-1} \Q[Y_n=\cdot\ | Y^{N\setminus n} = y^{N\setminus n}]}{
\langle\pi_{y_n}\odot \Pi^{-1} \Q[Y_n=\cdot\ | Y^{N\setminus n} = y^{N\setminus n}], \mathbf 1 \rangle},
\end{equation}
where $\Pi$ is the emission matrix, and $\pi_{-1},\pi_1$ are the columns of $\Pi$:
$$
\Pi=\begin{bmatrix} 1-\epsilon & \epsilon\\
\epsilon &1-\epsilon 
\end{bmatrix},
\quad\pi_{-1}=\begin{bmatrix} 1-\epsilon\\
\epsilon 
\end{bmatrix},
\quad\pi_{1}=\begin{bmatrix} \epsilon\\
1-\epsilon 
\end{bmatrix},
\quad\Pi^{-1}=\frac 1{1-2\epsilon} \begin{bmatrix}  
1-\epsilon & -\epsilon\\
-\epsilon &1-\epsilon
\end{bmatrix},
$$
and $\odot$ is componentwise product  of vectors of equal lengths,
$$
u \odot v = ( u_1\cdot v_1,\ldots, u_d\cdot  v_d).
$$
Expression (\ref{TWOSIDED}) opens a possibility of constructing denoisers
 when parameters of the underlying Markov chains are unknown; we continue to assume that the channel  remains known. Indeed, two-sided conditional
 probabilities $ \Q[Y_n=\cdot\ | Y^{N\setminus n} = y^{N\setminus n}]$ could be estimated from the data.  The Discrete Universal Denoiser (DUDE) \cite{Weissman}
 algorithm estimates
 conditional probabilities
\begin{equation}\label{DUDEestim}
 \Q( Y_n = c\,|\, Y_{n-k_N}^{n-1}= a_{-k_N}^{-1}, Y_{n+1}^{n+k_N}= b_{1}^{k_N})=
 \frac {m(a_{-k_N}^{-1},c,b_{1}^{k_N})}{\sum_{\bar c} m(a_{-k_N}^{-1},c,b_{1}^{k_N})}
\end{equation} 
 where $m(a_{-k_N}^{-1},c,b_{1}^{k_N})$ is the number of
   occurrences of the word $a_{-k_N}^{-1}cb_{1}^{k_N}$ in the observed sequence $y^N=(y_1,
 \ldots, y_N)$; the length of right and left contexts is set to $k_N = c\log N$,
 $c>0$.
 DUDE has shown excellent performance in a number of test cases.
In particular,  in caseof the binary memoryless channel and the symmetric
 Markov chain, considered in this paper, performance in 
 comparable to the one of the backward-forward
 algorithm \eqref{BF}, which requires full  knowledge of the source distribution, while DUDE is
 completely oblivious in that respect. 
In our opinion,  excellent performance of DUDE in this case is partially
due to the fact that $\Q$ is a Gibbs measure, admitting
smooth two-sided conditional probabilities, which are well approximated by \eqref{DUDEestim}
and thus can be estimated from the data.
It will be interesting to evaluate performance in cases when the output measure is
not Gibbs.

 Invention of DUDE sparked a great interest in two-sided approaches to information-theoretic problems. It turns out that despite the fact the
efficient algorithms for estimation of  one-sided models exist, the analogous 
two-sided problem is substantially more difficult. 
As alternatives to \eqref{DUDEestim}, other methods to estimate two-sided conditional
probabilities have been suggested , e.g., \cite{OWW2005, FVW2010,YV}. 
  For example,  Yu and Verd\'u \cite{YV} proposed a \mybf{Backward-Forward Product} (BFP) model:
 $$
 \widetilde\Q(y_0|y_{<0},y_{>0}) \propto \widetilde\Q(y_0|y_{<0}) \widetilde\Q(y_0|y_{>0}),
 $$
 and the one-sided conditional probabilities  $\widetilde\Q(y_0|y_{<0})$,
 $ \widetilde\Q(y_0|y_{>0})$ can be estimated using standard one-sided algorithms.
 Note, that in our model,
$$\begin{gathered}
\frac{\widetilde\Q( y_0| y_{<0}) \widetilde \Q( y_0| y_{>0})} 
  {\widetilde\Q( y_0| y_{<0}) \widetilde \Q( y_0| y_{>0})
  +\widetilde\Q( \bar y_0| y_{<0}) \widetilde \Q( \bar y_0| y_{>0})}\\
  =\frac { \cosh(Ky_0 +A(w_{-1}))\cosh(Ky_0 +A(w_{1}))}
   { \cosh(Ky_0 +A(w_{-1}))\cosh(Ky_0 +A(w_{1}))+  \cosh(-Ky_0 +A(w_{-1}))\cosh(-Ky_0 +A(w_{1}))}
  \end{gathered}
$$
 in general does not coincide with 
 $$\frac { \cosh\bigl(Ky_0+A(w_{-1})+A(w_1)\bigr)}
{ \cosh\bigl(Ky_0+A(w_{-1} )+A(w_1)\bigr)+ 
  \cosh\bigl(-Ky_0+A(w_{-1})+A(w_1)\bigr)}=\Q(y_0|y_{<0},y_{>0}).
$$
Nevertheless, the BFP model seems to perform extremely well \cite{YV}.

Among other alternatives, let us mention the possibility
to extend standard one-sided algorithms  to
 produce algorithms for estimating two-sided conditional probabilities from data.
 This approach is investigated in  \cite{Verb99b},
 where the densoising performance of the  resulting Gibbsian models is evaluated.
Gibbsian algorithm  performs better than DUDE:
bit error rates are given in the
table below for noise level $\epsilon=0.2$ and
various values of  $p$ (smaller rates are better). 

\begin{center}
\begin{tabular}{|c|c|c|}
\hline
 $p$  & Gibbs     
& DUDE
  \\
 \hline
 \hline
0.05 &  5.30\% &   5.58\%  \\
0.10 &   9.91\% &     10.48\%  \\
 0.15 &  13.20\% &  13.77\%  \\
0.20 &   18.34\% &   18.77\% \\
\hline
\end{tabular}
\end{center}

One could also try to estimate the Gibbsian potential directly,
e.g., using the estimation procedure proposed in \cite{EV}. This method 
showed promising performance in  experiments on language classification and authorship attribution. In conclusion, let us also mention that  the direct two-sided Gibbs modeling of
stochastic processes opens   possibilities for applying semi-parametric statistical procedures,
as opposed to the universal (parameter free) approach of DUDE.

\section*{Acknowledgments}
Part of the work described in this paper has been completed during author's visit to
the Institute of Mathematics for Industry, Kyushu University. The author is  grateful 
for the hospitality during his stay and the support of the World Premier International Researcher Invitation Program.

\end{document}